\documentclass{amsart}
\usepackage{amssymb
	,amsthm
	,amsmath,enumerate
	,amscd,tabularx,ragged2e,booktabs,caption
	,mathtools,
	hyperref,appendix
}
\usepackage[all]{xy}
\usepackage[top=30truemm,bottom=30truemm,left=25truemm,right=25truemm]{geometry}

\newcommand{\Z}{\mathbb{Z}}

\newcommand{\C}{\mathbb{C}}

\def\End{{\rm End}}
\def\M{{\rm M}}

\newcommand{\Irr}{\mathrm{Irr}}

\newcommand{\Res}{\mathrm{Res}}

\newcommand{\Hom}{\mathrm{Hom}}

\newcommand{\Sp}{\mathrm{Sp}}

\newcommand{\SL}{\mathrm{SL}}
\newcommand{\GL}{\mathrm{GL}}
\newcommand{\Mp}{\mathrm{Mp}}
\newcommand{\Oo}{\mathrm{O}}
\newcommand{\SO}{\mathrm{SO}}

\newcommand{\GSO}{\mathrm{GSO}}
\newcommand{\GSp}{\mathrm{GSp}}
\newcommand{\GSpin}{\mathrm{GSpin}}
\newcommand{\PGSp}{\mathrm{PGSp}}
\newcommand{\GO}{\mathrm{GO} }

\newcommand{\SU}{\mathrm{SU}}

\newcommand{\PGL}{\mathrm{PGL}}

\newcommand{\iif}{&\text{if }}
\newcommand{\other}{&\text{otherwise}}

\newtheorem{thm}{Theorem}[section]
\newtheorem{lem}[thm]{Lemma}

\newtheorem{coro}[thm]{Corollary}

\theoremstyle{remark}

\newtheorem{rem}[thm]{Remark}

\theoremstyle{definition}

\numberwithin{equation}{section}

\makeatletter
\def\iddots{\mathinner{\mkern1mu\raise\p@
		\hbox{.}\mkern2mu\raise4\p@\hbox{.}\mkern2mu
		\raise7\p@\vbox{\kern7\p@\hbox{.}}\mkern1mu}}
\def\adots{\mathinner{\mkern2mu\raise\p@\hbox{.}
		\mkern2mu\raise4\p@\hbox{.}\mkern1mu
		\raise7\p@\vbox{\kern7\p@\hbox{.}}\mkern1mu}}
\makeatother

\allowdisplaybreaks

\title{Some applications of theta correspondence to branching laws }
\author{Hengfei Lu}
\date{}
\address{Department of Mathematics, Weizmann Institute of Science, 234 Herzl St. P.O.B.26, Rehovot 7610001, Israel}
\email{hengfei.lu@weizmann.ac.il}

\begin{document}
	\maketitle
	\begin{abstract}
		In this paper, we will use the local theta correspondences for the dual pair $(\Sp(W),\Oo(V))$ to investigate some branching law problems. 
	\end{abstract}
\subsection*{Key words} the local theta lift, branching laws, disctinction problems
\subsection*{MSC(2000)} 11F27$\cdot$11F70$\cdot$22E50
	\tableofcontents
	\section{Introduction}
	Let $F$ be a nonarchimedean local field of characteristic zero. Let $G$ be a reductive group defined over $F$. Assume that $H$ is a closed subgroup of $G$ defined over $F$. Given an irreducible smooth representation $\pi$ of $G(F)$, there is a rich literature studying the restricted representation $\pi|_{H}$, called the branching law problems, such as \cite{GGP1,gan2018period,PD,GS,wd2012so,SV2017} and so on. If $$\dim\Hom_{H(F)}(\pi,\tau)\leq1$$ for any irreducible smooth representation $\pi$ (resp. $\tau$) of $G(F)$ (resp. $H(F)$), then the pair $(G(F),H(F))$ is called multiplicity-free, such as $(\SO_n,\SO_{n-1})$. If $\dim\Hom_{H(F)}(\pi,\mathbb{C}) \leq1$ for any irreducible smooth representation $\pi$ of $G(F)$, then the pair $(G(F),H(F))$ is called the Gelfand pair.
	  If $$\Hom_{H(F)}(\pi,\mathbb{C})\neq0,$$ then $\pi$ is called $H(F)$-distinguished.
	 The relative trace formula is a very powerful tool that can be applied to deal with the disctinction problems, such as \cite{wd2012so}. In this paper, we will use the local theta correspondence to deal with the disctinction problems, following \cite{GS,gan2018period}.
	
	Let $D$ be a $4$-dimensional division quaternion algebra over $F$. Let $E=F[\delta]$ be a quadratic field extension of $F$. Let $W_F$ (resp. $W_E$) be the Weil group of $F$ (resp. $E$). Let $WD_F$ (resp. $WD_E$) be the Weil-Deligne group of $F$ (resp. $E$). Fix a nontrivial additive character $\psi$ of $F$ and $\psi_E=\psi\circ tr_{E/F}$. Set $$\psi_0(e)=\psi_E(\delta e)=\psi(tr_{E/F}(\delta e)).$$ Then $\psi_0|_{F}$ is trivial. 
	
	There are $3$ aspects in this paper:
	\begin{itemize}
			\item to show that $(\GL_1(D)\times\GL_1(D),\GSpin_{4}^{1,1}(F))$ is multiplicity-free (see \S\ref{sec:GSpin});
		\item to compare the distinction problems for $(\Sp_{2n}(E),\Sp_{2n}(F))$ and $(\Oo(V\otimes_F E),\Oo(V) )$(see \S\ref{sec:SpO}) ;
		\item to classify all the tempered representations of $\SO(V)$, distinguished by $\SO(V')$, where $V'$ is a subspace of $V$ of codimension $1$.
	\end{itemize}
Let $\GSpin_{4}^{1,1}(F)=\{(g_1,g_2)\in\GL_1(D)\times\GL_1(D):\det(g_1)=\deg(g_2) \}$, where $\det(g_1)=N_{D/F}(g_1)$. 
	\begin{thm}\label{thmGSpin}
		Given an irreducible  representation $\pi=(\pi_1,\pi_2)$ of $\GL_1(D)\times\GL_1(D)$ and any irreducible representation $\tau$ of $\GSpin_{4}^{1,1}(F)$, one has
		\[\dim\Hom_{\GSpin_{4}^{1,1}(F)}(\pi,\tau)\leq1.  \]
	\end{thm}
We would like to highlight the fact that Asgari-Choiy \cite{choiy2017} have proved the local Langlands conjectures for $\GSpin_{4}=\{(g_1,g_2)\in\GL_2\times\GL_2:\det(g_1)=\det(g_2) \}$ and its inner forms. However, in \cite[Remark 5.11]{choiy2017}, they claim that $|\Pi_{\phi}(\GSpin_4^{1,1}(F)) |=1$ if and only if $\pi_1=\pi_2\otimes\chi$ are dihedral with respect to three nontrivial quadratic characters of $F^\times$, where $\chi$ is any character of $F^\times$. In this case, they obtain the conclusion that the multiplicity in restriction from $\GL_1(D)\times\GL_1(D)$ to $\GSpin_{4}^{1,1}(F)$ could be $2$, which is false. In \S\ref{sec:GSpin}, we will show that $|\Pi_{\phi}(\GSpin_{4}^{1,1}(F)) |=4$ instead of $1$ if $\pi_1=\pi_2\otimes\chi$ are dihedral with respect to three nontrivial quadratic characters of $F^\times$. On the other hand, Dipendra Prasad gives an another proof of Theorem \ref{thmGSpin} involving pure algebras which will be presented at the end of \S\ref{sec:GSpin}.
	\begin{thm}\label{PrasadSpO}
		Let $E=F[\delta]$ be a quadratic field extension over a nonarchimedean local field $F$ of characteristic zero. Let $W_n$ be a $2n$-dimensional symplectic vector space over $F$.
		Let $\tau$ be a tempered  representation of $\Sp_{2n}(E)$. The local theta lift to $\Oo_{r,r}(E)$ of $\tau$ from $\Sp_{2n}(E)$ is denoted by $\theta_{\psi_0}^{r,r}(\tau)$.
		\begin{enumerate}[(i)]
						\item Suppose that the local theta lift $\theta^{n,n}_{\psi_0}(\tau)=0$. Then $\tau$ is $\Sp(W_n)$-distinguished if and only if $\theta_{\psi_0}^{n+1,n+1}(\tau)$ is nonzero and $\Oo_{n+1,n+1}(F)$-distinguished.
			\item 	Suppose that $\theta^{n,n}_{\psi_0}(\tau)$ is a nonzero representation of $\Oo_{n,n}(E)$ and that the local theta lift to $\Sp_{2n-2}(E)$ of $\theta^{n,n}_{\psi_0}(\tau)$ is zero. Then $\tau$ is $\Sp_{2n}(F)$-distinguished if and only $\theta^{n,n}_{\psi_0}(\tau)$ is $\Oo(V)$-distinguished for a $2n$-dimensional quadratic  space $V$ over $F$ satisfying
			$\Oo(V\otimes_F E)=\Oo_{n,n}(E)$.
		\end{enumerate}
	\end{thm}
	In \cite{zhang2018theta}, assuming $p\neq 2$, Zhang showed that given a $\Sp_{2n}(F)$-distinguished regular supercuspidal representation $\tau$ of $\Sp_{2n}(E)$, there exists a $2n$-dimensional quadratic space $V$ over $F$ such that the theta lift $\theta_{\psi_0}(\tau)$ is $\Oo(V)$-distinguished. Moreover, $V\otimes_F E$ is unique. We extend Zhang's result from the regular supercuspidal representations to the tempered representations of $\Oo_{n,n}(E)$, whose first occurence index (defined in \S2) is $2n$ in the Witt tower $\{W_n\otimes_F E \}$, including the case when $p=2$.
	
	Finally, we use the twisted Jacquet modules of the Weil representation to classify all the tempered representations of $\SO(V)$ distinguished by $\SO(V')$ with $\dim(V/V')=1$, which turns out to be the local theta lifts from $\SL_2(F)$ or $\Mp_2(F)$ to $\SO(V)$. The key ingredient is that the local theta lifts to $\SO(V)$ from $\SL_2(F)$ or $\Mp_2(F)$ must not be tempered when $\dim V\geq 8$. For a small group $\SO(V)$ when $\dim V\leq 7$, the local theta lift from $\SL_2(F)$ or $\Mp_2(F)$ can be written down explicitly in terms of the Langlands parameter due to \cite[Theorem 4.1]{AG}.
	And we give a proof to \cite[Conjecture 2]{PD} in the case of non-archimedean fields.
	
	The paper is organized as follows. In \S $2$, we set up the notation about the local theta correspondence. We will review the results of Asgari-Chioy and then use the local theta correspondence to prove
	 Theorem \ref{thmGSpin}   in \S\ref{sec:GSpin}. Another proof of Theorem \ref{thmGSpin} due to Dipendra Prasad will be given at the end of \S\ref{sec:GSpin}. In \S\ref{sec:SpO}, we will show that $\tau$ is $\Sp_{2n}(F)$-distinguished if and only if the local theta lift is $\Oo(V)$-distinguished for a quadratic space $V$ over $F$. In \S$5$,  we will  classify all the tempered representations of $\SO(V)$ distinguished by $\SO(V')$ with $\dim(V/V')=1$ and then prove \cite[Conjecture]{PD} in the case of non-archimedean fields. In the appendix, we review the local theta correspondence of  tempered representations and Langlands parameters.
	
	\paragraph*{\textbf{Acknowledgments}} The author is grateful to Wee Teck Gan and Dipendra Prasad for their guidance and numerous discussions. He also wants to thank Sandeep Varma for helpful discussions. He would like to thank the referee for useful comments as well. This research was done when he was a Visiting Fellow at Tata Institute of Fundamental Research, Mumbai and it was partially supported by ERC StG grant number 637912 under revision.
	
	\section{The local theta correspondence}
	In this section, we will briefly recall some results about the local theta correspondence, following \cite{kudla1996notes}.  
	
	Let $F$ be a local field of characteristic zero.
	Consider the dual pair $\Oo(V)\times \Sp(W)$.
	For simplicity, we may assume that $\dim V$ is even. Fix a nontrivial additive character $\psi$ of $F.$
	Let $\omega_\psi$ be the Weil representation for $\Oo(V)\times \Sp(W).$ 
	If $\pi$ is an irreducible representation of $\Oo(V)$ (resp. $\Sp(W)$), the maximal $\pi$-isotypic quotient of the Weil representation $\omega_{\psi}$ has the form 
	\[\pi\boxtimes\Theta_\psi(\pi) \]
	for some smooth representation of $\Sp(W)$ (resp. some smooth representation of $\Oo(V)$). We call $\Theta_\psi(\pi )$
	the big theta lift of $\pi$. Let $\theta_\psi(\pi)$ be the maximal semisimple quotient of $\Theta_\psi(\pi),$ which is called the small theta lift of $\pi.$ 
	\begin{thm} \cite{gan2014howe,gan2014proof} One has
		\begin{enumerate}[(i)]
			\item $\theta_\psi(\pi)$ is irreducible whenever $\Theta_\psi(\pi)$ is non-zero.
			\item the map $\pi\mapsto \theta_\psi(\pi)$ is injective on its domain.
		\end{enumerate}
	\end{thm}
	It is called the Howe duality conjecture which has been proven by Waldspurger \cite{waldspurger1990demonstration} when $p\neq2$.
	
	Then Roberts \cite{roberts1996similitudes} extends the Weil representation to the case of similitude groups $\GO(V)\times\GSp(W)$ (the similitude character  $\lambda_V$ is surjective in this paper).
	\begin{lem}
		\cite[Lemma 2.2]{gan2011locallanglands} Let $\tau$ be an irreducible representation of $\GSp(W)$ and $\tau|_{\Sp(W)}=\oplus_i\tau_i$. Then
		\[\Theta(\tau)|_{\Oo(V)}=\oplus_i\Theta_\psi(\tau_i). \]
	\end{lem}
	
	\subsection{First occurence indices for pairs of orthogonal Witt towers} Let $W_n$ be the $2n$-dimensional symplectic vector space over $F$ with associated metaplectic group $\Sp(W_n)$ and consider the two towers of orthogonal groups attached to the quadratic spaces with trivial discriminant. More precisely, 
	 let $V^+_{2r}$ (resp. $V_{2r}^-$) be the $2r$-dimensional quadratic $F$-vector space with trivial discriminant and Hasse invariant $+1$ (resp. $-1$).
	Denote the orthogonal groups by $\Oo(V_{2r}^+)=\Oo_{r,r}(F)$ (resp. $\Oo(V_{2r}^-)=\Oo_{r+2,r-2}(F)$). For an irreducible smooth representation $\pi$ of $\Sp(W_n),$ one may consider the theta lifts $\theta^{r,r}_\psi(\pi)$ and $\theta^{r+2,r-2}_\psi(\pi)$ to
	$\Oo(V^+_{2r})$ and $\Oo(V_{2r}^-)$ respectively, with respect to a fixed non-trivial additive character $\psi$. Set
	\[\begin{cases}
	r^+(\pi)=\inf\{2r:\theta^{r,r}_\psi(\pi)\neq0 \};\\
	r^-(\pi)=\inf\{2r:\theta^{r+2,r-2}_\psi(\pi)\neq0 \}.
	\end{cases} \]
	Then Kudla-Rallis \cite{kudla2005first} and Sun-Zhu \cite{sun2012conservation} showed:
	\begin{thm}
		[Conservation Relation] For any irreducible representation $\pi$ of $\Mp(W_n),$ we have
		\[r^+(\pi)+r^-(\pi)=4n+4=4+2\dim W_n. \]
	\end{thm}
In \cite{AG}, Atobe and Gan use $m^{\mbox{down}}(\pi)$ (resp. $m^{\mbox{up}}(\pi)$) to denote $\min\{r^+(\pi),r^-(\pi)\}$ (resp.  $\max\{r^+(\pi),r^-(\pi) \} $).
\subsection{The see-saw diagrams}
A pair $(G_1,H_1)$ and $(G_2,H_2)$ of reductive dual pairs in a symplectic group is called the see-saw pair if $H_1\subset G_2$ and $H_2\subset G_1$.
\begin{lem}\cite[Page 6]{prasad1996branching}
	For a see-saw pair of reductive dual pairs $(G_1,H_1)$ and $(G_2,H_2)$, let $\pi_1$ (resp. $\pi_2$) be a representation of $H_1$ (resp. $H_2$). Then we have the following isomorphism
	\begin{equation}\label{see-sawsimilitude}
	\Hom_{H_1}(\Theta_\psi(\pi_2),\pi_1)\cong \Hom_{H_2}(\Theta_\psi(\pi_1),\pi_2). 
	\end{equation} 
\end{lem}
Dipendra Prasad proved that the see-saw identity \eqref{see-sawsimilitude} holds for the similitude groups as well.

	 \section{The $L$-packet for $\GSpin_{4}$}\label{sec:GSpin}
	 This section focuses on the proof of Theorem \ref{thmGSpin}.
Let $$\GSpin_4(F)=\{(g_1,g_2)\in\GL_2(F)\times\GL_2(F):\det(g_1)=\det(g_2) \}$$ be a subgroup of $\GL_2(F)\times\GL_2(F)$.	 Asgari-Choiy \cite{choiy2017} used the restriction from $\GL_2\times\GL_2$ to $\GSpin_{4}$ to set up the Langlands correspondence for the generic representations of $\GSpin_{4}$. We summarize their results as follow.

Let $\tau$ be an irreducible smooth generic representation of $\GSpin_{4}(F)$. There exists a unique enhanced $L$-parameter $(\phi,\lambda)$ such that
\[\phi:WD_F\longrightarrow{}^L\GSpin_{4}=\GO_4(\mathbb{C}), \]
where $\lambda$ is a character of the component group $S_\phi=C(\phi)/C^\circ(\phi)$, where $C(\phi)$ is the centralizer of $\phi$ in $\GO_{4}(\mathbb{C})$. Let $\tilde{\tau}$ be a representation of $\GL_2(F)\times\GL_2(F)$ such that $\tilde{\tau}|_{\GSpin_{4}(F)}\supset \tau$. Then the size of the $L$-packet $\Pi_{\phi} $ is
\[|\Pi_{\phi}|=|\{(\chi,\chi):(\chi,\chi)\otimes\tilde{\tau}=\tilde{\tau} \} | \]
where $\chi:F^\times\longrightarrow\mathbb{C}^\times$ is a quadratic character of $F^\times$. 

For the inner forms of $\GSpin_{4}$, we consider the extended component group $\tilde{S}_{\phi}$, where $\tilde{S}_\phi$ sits in the following exact sequence
\[\xymatrix{1\ar[r]&\mu_2\times\mu_2\ar[r]& \tilde{S}_\phi\ar[r]&S_{\phi}\ar[r]&1 .} \]
Let $\xi_4,\xi_4^{i,j}$ be the characters of $\mu_2\times\mu_2$, where $(i,j)=(1,1),(1,2)$ or $(2,1)$. We may also regard $\xi_4$ as $\xi_4^{2,2}$ and $\GSpin_{4}=\GSpin_{4}^{2,2}$.
\begin{thm}
	\cite[Theorem 5.1]{choiy2017} Given an $L$-parameter $\phi$ of $\GSpin_{4}(F)$, there exists a bijection
$$	 \xymatrix{\Pi_{\phi}(\GSpin_{4}^{i,j})\ar[r]& \Irr(\tilde{S}_\phi,\xi_4^{i,j})\ar[l]  }$$
where $\Irr(\tilde{S}_\phi,\xi_4^{i,j})$ are the set of irreducible representations of $\tilde{S}_\phi$ which extends $\xi_4^{i,j}$, $$\GSpin_{4}^{1,2}(F)=\{(g_1,g_2)\in\GL_1(D)\times\GL_2(F):\det(g_1)=\det(g_2) \}$$
and $\GSpin_{4}^{2,1}(F)=\{(g_1,g_2)\in\GL_2(F)\times\GL_1(D):\det(g_1)=\det(g_2) \}$.
\end{thm}
Suppose that $\pi=(\pi_1,\pi_2)$ is a representation of $\GSpin_{4}^{1,1}(F)$ with $\phi=(\phi_1,\phi_2)$.
If $\phi_1=\phi_2\otimes\chi$, then
\[\Pi_{\phi}(\GSpin_{4})=\begin{cases}
1,&\mbox{ if }\phi_1\mbox{ is primitive or nontrivial on }\SL_2(\mathbb{C});\\
2,&\mbox{ if }\phi_1\mbox{ is dihedral with repect to one quadratic character};\\
4,&\mbox{ if }\phi_1\mbox{ is dihedral with respect to three quadratic characters.}
\end{cases} \]
If $\phi_1\neq\phi_2\otimes\chi$ for any character $\chi$ of $F^\times$, then $|\Pi_{\phi}(\GSpin_{4})|\leq2$.

Let $$\GSO_{2,2}(F)=\frac{\GL_2(F)\times\GL_2(F)}{\{(t,t),t\in F^\times\}}\mbox{    and    } \SO_{2,2}(F)\cong\frac{\GSpin_{4}(F)}{\{(t,t),t\in F^\times\}}.$$ Given a square-integrable representation $\rho$ of $\GL_2(F)$, then the theta lift $\theta_\psi(\rho)$ from $\GL_2(F)$ to $\GSO_{2,2}(F)$ is isomorphic to $\rho\boxtimes\rho^\vee$. Similarly, the theta lift to $\GSO_{4,0}(F)$ from $\GL_2(F)$ is
\[\theta_\psi(\rho)=JL(\rho)\boxtimes(JL(\rho))^\vee \]
where $JL(\rho)$ is the Jacquet-Langlands correspondence representation of $\GL_1(D)$ associated to $\rho$ and
\[\GSO_{4,0}(F)\cong\frac{\GL_1(D)\times\GL_1(D)}{\{(t,t):t\in F^\times \} } .\]
Now we start to prove Theorem \ref{thmGSpin}.
\begin{proof}[Proof of Theorem \ref{thmGSpin}]  Assume that $\pi=(\pi_1,\pi_2)$ is an irreducible representation of $\GL_1(D)\times\GL_1(D)$ with $L$-parameter $(\phi_1,\phi_2)$. We divide them into two cases:
\begin{itemize}
	\item If $|\Pi_{\phi}(\GSpin_4)|\leq2$, then due to Schur orthogonal relation, $|\Pi_{\phi}(\GSpin_{4})|=|\Pi_{\phi}(\GSpin_{4}^{1,1})|$ and $\pi|_{\GSpin_{4}^{1,1}(F)}$ is multiplicity free.
	\item If $|\Pi_{\phi}(\GSpin_{4})|=4$, then $\phi_1=\phi_2\otimes\chi$ for a character $\chi$ of $W_F$. Without loss of generality, we assume that $\chi=\det\phi_1$. Then $\phi_2=\phi_1^\vee$. Thus $\pi$ corresponds to a representation of $\GSO_{4,0}(F)$ and $\pi=\theta_\psi(JL(\pi_1))$, where $JL(\pi_1)$ is the Jacquet-Langlands lift of $\pi_1$. If $\Hom_{\GSpin_{4}^{1,1}(F)}(\pi,\tau)\neq0 $ for a representation $\tau$ of $\GSpin_{4}^{1,1}(F)$, then $\tau$ is a representation of $\SO_{4,0}(F) $.
	Thanks to \cite[\S6]{adler2018multiplicity}, the restriction from $\GSO(n)$ to $\SO(n)$ is multiplicity-free. Then $\dim\Hom_{\GSpin_{4}^{1,1}(F)}(\pi,\tau)=1$. In fact, 
	\[JL(\pi_1)|_{\SL_2(F)}=\tau_1\oplus\tau_2\oplus\tau_3\oplus\tau_4 \]
	and so $\pi|_{\GSpin_{4}^{1,1}(F)}=\oplus_ip^\ast(\theta_\psi(\tau_i))$, where $p:\GSpin_{4}^{1,1}(F)\longrightarrow\SO_{4,0}(F) $ is the natural projection and $p^\ast$ is the pull back. In general,
	\[\pi|_{\GSpin_{4}^{1,1}(F)}=\bigoplus_i p^\ast(\theta_\psi(\tau_i))\otimes(\mathbf{1},\chi^{-1}\det\phi_1)|_{\GSpin_{4}^{1,1}(F)}. \]
	 Therefore $|\Pi_{\phi}(\GSpin_{4}^{1,1}) |=4$ and $\pi|_{\GSpin_{4}^{1,1}(F)} $ is  multiplicity-free.
\end{itemize}
\end{proof}
\begin{rem}
	It is known that $\pi_1|_{\SL_1(D)}=2\tau_0$ is of multiplicity two if $\pi_1$ is dihedral with respect to three quadratic character, where $\tau_0$ is an irreducible representation of $\SL_1(D)$. However, the representation of two copies of $\GL_1(D)$ restricts to the subgroup $\GSpin_{4}^{1,1}(F)$ is multiplicity-free.
\end{rem}

\begin{rem}
	One can use the theta correspondence for the similitude quaternionic dual pair to show that
	\[|\Pi_{\phi}(\GSpin_{4}^{1,2}(F))|=1 \]
	when $$\phi=\chi(\phi_2\otimes\phi_2)=\chi\oplus\chi\chi_1\oplus\chi\chi_2\oplus\chi\chi_1\chi_2 :WD_F\longrightarrow\GO_4(\mathbb{C}),$$ where $\chi_1$ and $\chi_2$ are quadratic characters of $W_F$. In this case, let $\pi'=(\pi_2\otimes\chi,JL(\pi_2))$ be a representation of $\GL_1(D)\times\GL_2(F) $. Then the restricted representation $\pi'|_{\GSpin_{4}^{1,2}(F)}$ is of multiplicity two.
\end{rem}
\begin{rem} Let $\phi_1=\phi_2\otimes\chi$ and $\phi_1$ is dihedral with respect to three quadratic characters.
	On the parameter side, $S_{\phi}$ is a finite group of order $4$. The extended component group 
	$\tilde{S}_\phi$ is isomorphic to $\langle a\rangle\times Q_8$ where $a^2=1$ and $Q_8=\{\pm1,\pm i,\pm j,\pm ij \}$ is the
	quaternion group of order $8$.  The center of $\tilde{S}_\phi$ is $\langle a\rangle\times\langle-1\rangle$. Note that $\xi_4^{1,1}(-1)=1$ and $\xi_4^{1,1}(a)=-1$. Then there are four ways to extend the character $\xi_4^{1,1}$ to $\tilde{S}_\phi$ and so $$|\Irr(\tilde{S}_\phi,\xi_4^{1,1})|=4.$$
\end{rem}

We give an another proof of Theorem \ref{thmGSpin} due to Dipendra Prasad.
\begin{proof}[Another proof of Theorem \ref{thmGSpin}]
Recall that the essential case of the theorem is to prove that if $\pi \otimes \pi$
is an irreducible representation of $\GL_1(D) \times \GL_1(D)$ where $\pi$ is an irreducible
representation of $\GL_1(D)$ with two distinct quadratic self-twists, then the restriction
of  $\pi \otimes \pi$ to  $$\GSpin^{1,1}_4(F)= \{ (g_1,g_2) \in \GL_1(D) \times \GL_1(D) | \det g_1 = \det g_2\}$$
decomposes as a sum of 4 distinct irreducible representations of $\GSpin^{1,1}_4(F)$. The representation
$\pi$ (with two distinct quadratic twists) will be fixed in what follows. We begin by recalling
what the standard theory says about the  restriction of $\pi$ from $\GL_1(D)$ to $\SL_1(D)$.

It is known that $\pi$ restricted to $\SL_1(D)$ is $2\pi'$ for an irreducible representation $\pi'$
of $\SL_1(D)$. It follows that
$$\End_{\SL_1(D)}(\pi) \cong \M_2(\C),$$
and the natural action of $\GL_1(D)$ on $\End_{\SL_1(D)}(\pi) \cong \M_2(\C),$ is via algebra automorphisms
of $\M_2(\C)$, and thus gives a projective representation of $\GL_1(D)/\SL_1(D) \cong F^\times \rightarrow \PGL_2(\C)$.
This projective representation of $F^\times$ factors through the quotient of $F^\times$  by the common
kernel of the quadratic self-twists of $\pi$, and thus  $\GL_1(D)/\SL_1(D)$ acts through a quotient which is
$\Z/2 + \Z/2$, giving rise to a projective faithful representation of $\Z/2 + \Z/2$ into   $\PGL_2(\C)$. Now there is a unique
subgroup of $\PGL_2(\C)$ isomorphic to $\Z/2 + \Z/2$, consisting of the image of the diagonal matrix $(-1,1)$ in $\GL_2(\C)$, and an element
in the Weyl group of the diagonal torus. (The resulting map from $F^\times$ to $\PGL_2(\C)$ is the Langlands parameter of $\pi$ with values in $\PGL_2(\C)$, although this fact has no relevance for us here.)

The projective representation of $\GL_1(D)/\SL_1(D) \cong F^\times$ into $\PGL_2(\C)$ could be considered as the projective
representation attached to the unique 2-dimensional irreducible representation $V_2$ of $Q_8$, the
quaternionic group of order 8.

Observe that $$\End_{\SL_1(D) \times \SL_1(D)}(\pi \otimes \pi) \cong \End_{\GL_1(D)}( \pi) \otimes \End_{\GL_1(D)}( \pi) \cong
\M_2(\C) \otimes \M_2(\C),$$
and that, $\End_{\SL_1(D) \times \SL_1(D)}(\pi \otimes \pi)$ comes equipped with an action of $\GL_1(D)/\SL_1(D) \times \GL_1(D)/\SL_1(D) \cong F^\times  \times F^\times$ by algebra automorphisms, which can be realized as the representation of $Q_8 \times Q_8$
on the 4-dimensional irreducible representation $V_2 \boxtimes V_2$ of it.

To prove that $\pi\otimes \pi$ restricted to $\GSpin^{1,1}_4(F)$ decomposes as a sum of 4 distinct irreducible representations of $\GSpin^{1,1}_4(F)$, we need to prove that 
$$\End_{\GSpin^{1,1}_4(F)}(\pi \otimes \pi) = \C+\C+\C+\C,$$
as algebras. However, $\End_{\GSpin_4^{1,1}(F)}(\pi \otimes \pi)$ is simply $\Delta Q_8 \subset Q_8 \times Q_8$ invariants in
$\End_{\SL_1(D) \times \SL_1(D)}(\pi \otimes \pi)$ which is now the commutant of the action of $\Delta Q_8$ on the
4-dimensional representation $V_2 \otimes V_2$, which being the sum of 4 characters of $Q_8$, the commutant is
simply $ \C+\C+\C+\C,$ as desired.
\end{proof}

	 \section{The disctinction problems for $\Sp(W_n)$ and $\Oo(V)$}\label{sec:SpO}
In this section, we will study the disctinction problems for $\Sp(W_n)$ and $\Oo(V)$ over a quadratic field extension $E/F$.	 We will use $V_E$ (resp. $V_F$ or $V$) to denote the quadratic space over $E$ (resp. $F$).
	\subsection{Local Siegel-Weil identity} Let $(V_E,q_E)$ be a quadratic vector space over $E$.
	Let $V_F=\Res_{E/F}V_E$ be the same vector
	space  but now thought of as a vector space over $F$ with a quadratic form
	\[q_F(v)=\frac{1}{2}tr_{E/F}\circ q_E(v). \]
	If $W_n$ is a symplectic vector space over $F$ with symplectic group $\Sp_{2n}(F)$, then $W_n \otimes_F E$ is a symplectic vector space over $E$ with symplectic group $\Sp_{2n}(E)$. Then we have the following isomorphism of symplectic spaces:
	\[\Res_{E/F}(V_E\otimes (W_n\otimes_FE))=W_n\otimes V_F:=\mathbf{W}. \]
	The pair $(\Oo(V_E),\Sp_{2n}(E))$ and $(\Oo(V_F),\Sp_{2n}(F))$ of dual pairs in $\Sp(\mathbf{W})$ is a see-saw pair. 
	Given an irreducible representation $\Sigma$ of $\Oo(V_E)$, one has
		\[\dim\Hom_{\Oo(V_E)}(\Theta_\psi(\mathbf{1}),\Sigma )=\dim\Hom_{\Sp_{2n}(F)}(\Theta_{\psi_E}(\Sigma),\mathbb{C} )   \]
		where $\Theta_\psi(\mathbf{1})$ (resp. $\Theta_{\psi_E}(\Sigma)$) is the big theta lift to $\Oo(V_F)$ (resp. $\Sp_{2n}(E)$) of the trivial representation $\mathbb{C}$ (resp. the representation $\Sigma$) of $\Sp_{2n}(F)$ (resp. $\Oo(V_E)$).
	It is called the local Siegel-Weil identity which will be used in
	the proof of Theorem \ref{PrasadSpO} in the next subsection.
	
	%
\subsection{Proof of Theorem \ref{PrasadSpO}}
Now we can give the proof to Theorem \ref{PrasadSpO}.
\begin{proof}[Proof of Theorem \ref{PrasadSpO}]
Assume that $\tau$ is a tempered representation of $\Sp_{2n}(E)$. Let $\tau^\delta$ be the conjugate representation of $\tau$, i.e.
\[\tau^\delta(g)=\tau(g_\delta^{-1}gg_\delta) \]
where $g_\delta\in\GSp_{2n}(E)$ with similitude $\lambda_W(g_\delta)=\delta$.
\begin{enumerate}[(i)]
	\item If $\theta^{n+1,n+1}_{\psi_0}(\tau)$ is zero, then $\theta_{\psi_0}^{n+2,n-2}(\tau)\neq0$ by the conservation relation. Suppose that $V_E$ is a $2n$-dimensional quadratic space over $E$ with non-quasi-split othogonal group $\Oo(V_E)=\Oo_{n+2,n-2}(E)$. The theta lift to $\Oo(V_E)$ of $\tau$ is nonzero.
	 Due to \cite[Lemma 4.2.2]{hengfei2017}, the discriminant of the $4n$-dimensional $F$-vector space $V_F=\Res_{E/F}V_E$ is trivial and the Hasse invariant is $-1$. Consider the following see-saw diagram
	\[\xymatrix{\Sp_{2n}(E)\ar@{-}[d]\ar@{-}[rd] &\Oo(V_F)\ar@{-}[ld]\ar@{-}[d] \\ \Sp_{2n}(F)&\Oo(V_E) } \]
	where $\Oo(V_F)\cong \Oo_{2n+2,2n-2}(F)$. By the conservation relation, $\Theta^{2n+2,2n-2}_\psi(\mathbf{1})=0$. Then
	\[\dim\Hom_{\Sp_{2n}(F)}(\tau,\mathbb{C})=\dim\Hom_{\Oo(V_E)}(\Theta^{2n+2,2n-2}_\psi(\mathbf{1}),\Sigma)=0,   \]
	where $\Theta^{2n}_{\psi_E}(\Sigma)=\tau$ is the theta lift to $\Sp_{2n}(E)$ of $\Sigma$. Hence $\tau$ is not $\Sp(W_n)$-distinguished. 
	
	Suppose that $\theta_{\psi_0}^{n+1,n+1}(\tau)$ is nonzero. Let us consider the following
	see-saw diagrams
	\[\xymatrix{\Oo_{n+1,n+1}(E)\ar@{-}[d]\ar@{-}[rd] &\Sp_{4n}(F)\ar@{-}[d]\ar@{-}[ld]\ar@{-}[rd]&\Oo_{n,n}(E)\ar@{-}[d]\ar@{-}[ld]\\
		\Oo_{n+1,n+1}(F)&\Sp_{2n}(E)&\Oo_{n+2,n-2}(F)  } \]
	with $\Sigma=\theta_{\psi_0}^{n+1,n+1}(\tau)$. By the assumptions, $$\theta_{\psi_0}^{n,n}(\tau)=\theta_{\psi_E}^{n,n}(\tau^\delta)=0.$$ Note that $\theta^{n+1,n+1}_{\psi_E}(\tau^\delta)=\Sigma=\theta^{n+1,n+1}_{\psi_0}(\tau)$. There exists a short exact sequence 
	\[\xymatrix{0\ar[r]&R^{n+1,n+1}(\mathbf{1})\ar[r]& I(1/2)\ar[r]& R^{n+2,n-2}(\mathbf{1})\ar[r]&0 } \]
	of $\Sp_{4n}(F)$-modules,
	where $R^{i,j}(\mathbf{1})$ is the big theta lift of the trivial representation from $\Oo_{i,j}(F)$ to $\Sp_{4n}(F)$ and $I(s)$ is the degenerate principal series of $\Sp_{4n}(F)$ (see \cite[Proposition 7.2]{gan2014formal}) induced from a character $|\det|^s$ on the Siegel parabolic subgroup $Q$.
	Then 
	\[\dim\Hom_{\Sp_{2n}(E)}(I(1/2),\tau^\delta)\leq\dim\Hom_{\Oo_{n+1,n+1}(F)}(\Sigma,\mathbb{C})   .\]
	There exists a $\Sp_{2n}(E)$-equivariant filtration 
	\[0\subset I_0(s)\subset I_1(s)\subset \cdots \subset I_{n-1}(s)\subset I_{n}(s)=I(s)|_{\Sp_{2n}(E)} \]
	of $I(s)|_{\Sp_{2n}(E)}$
	such that $I_0(s)\cong ind_{\Sp_{2n}(F)}^{\Sp_{2n}(E)}\mathbb{C}$ and
	$I_{i}(s)/I_{i-1}(s) \cong ind_{Q_i}^{\Sp_{2n}(E)}(|\det|_E^{s+i/2} \otimes \mathbf{1})$ for $i = 1,2,\cdots ,n$. Here
	$Q_i \cong (\GL_i(E)\times \Sp_{2n-2i}(F))\cdot(Mat_{2i,2n-2i}(F)\times Sym^i (E))$, where $Mat_{m,n}(F)$ is the matrix space consisting of all $m \times n$ matrices and $Sym^i(E)$ consists of symmetric matrices in $Mat_{i,i}(E)$.
	Since $\tau$ is tempered, $\tau^\delta$ is tempered.
	Observe that $\Sp_{2n}(E)$ is the fixed point of a involution on $\Sp_{4n}(F)$, which is given by the scalar matrix $$g=\sqrt{d}\in \GSp_{2n}(E)\subset \GSp_{4n}(F)$$ acting on $\Sp_{4n}(F)$ by conjugation.
	Due to \cite[Theorem 2.5]{olafsson1987fourier}, there exists a polynomial $f$ on $\Sp_{4n}(F)$ such that 
	the complements of the open orbits in the double coset $Q\backslash \Sp_{4n}(F)/\Sp_{2n}(E)$ is the zero set of $f$. Thanks to \cite[Proposition 4.9]{dima2018analytic}, the multiplicity $\dim\Hom_{\Sp_{2n}(E)}(I(1/2),\tau^\delta) $ is at least $\dim\Hom_{\Sp_{2n}(E)}(I_0(1/2),\tau^\delta)$ where the submodule $I_0$ corresponds to the open orbits. More precisely,
	\[\dim\Hom_{\Sp_{2n}(F)}(\tau,\mathbb{C})=\dim\Hom_{\Sp_{2n}(E)}(I_0(1/2),\tau^\delta)\leq\dim\Hom_{\Sp_{2n}(E)}(I(1/2),\tau^\delta)\leq \dim\Hom_{\Oo_{n+1,n+1}(F)}(\Sigma,\mathbb{C}). \]
	Then it suffices to show that $\dim\Hom_{\Sp_{2n}(F)}(\tau,\mathbb{C})\geq \dim\Hom_{\Oo_{n+1,n+1}(F)}(\Sigma,\mathbb{C}) $.
	
	Due to \cite[Proposition 7.2]{gan2014formal}, there are two exact sequences of $\Sp_{4n}(F)$-modules
	\[\xymatrix{0\ar[r]&R^{n,n}(\mathbf{1})\oplus R^{n+2,n-2}(\mathbf{1})\ar[r]& I(-1/2)\ar[r]& R^{n+1,n+1}(\mathbf{1})\cap R^{n+3,n-1}(\mathbf{1})\ar[r]&0 } \]
	and
	\[\xymatrix{0\ar[r]&R^{n+1,n+1}(\mathbf{1})\cap R^{n+3,n-1}(\mathbf{1})\ar[r]& R^{n+1,n+1}(\mathbf{1})\ar[r]& R^{n,n}(\mathbf{1})\ar[r]&0 }. \]
	Then $\dim\Hom_{\Sp_{2n}(E)}(I(-1/2),\tau^\delta)=\dim\Hom_{\Sp_{2n}(E)}(R^{n+1,n+1}(\mathbf{1})\cap R^{n+3,n-1}(\mathbf{1}),\tau^\delta)$ and \[\dim\Hom_{\Sp_{2n}(E)}(R^{n+1,n+1}(\mathbf{1})\cap R^{n+3,n-1}(\mathbf{1}),\tau^\delta) \geq\dim\Hom_{\Sp_{2n}(E)}(R^{n+1,n+1}(\mathbf{1}),\tau^\delta). \]
	Note that $\tau^\delta$ does not occur on the boundary of $I(-1/2)$. If not, assuming 
	\[\dim\Hom_{\Sp_{2n}(E)}(I_i(s_0)/I_{i-1}(s_0),\tau^\delta )  \]
	for some $i$ and $s_0=-1/2$, then 
	\[\Hom_{\GL_i(E) }(|\det |_E^{s_0+i/2},R_{\bar{Q}'_i}(\tau^\delta) ) \neq0 \]
	with $s_0+i/2\geq0$, where ${Q}'_i$ is a maximal parabolic subgroup of $\Sp_{2n}(E)$ with Levi subgroup $\GL_i(E)\times\Sp_{2n-2i}(E)$ and $R_{\bar{Q}'_i}$ is the Jaquect functor with respect to the opposite parabolic $\bar{Q}'_i$. Due to Casselman's criterion for the tempered representation $\tau^\delta$, $s_0+i/2=0$ and so $i=1$. Moreover, the $L$-parameter of $\tau$ is given by 
	$2\mathbb{C}\oplus\phi'$ with $\phi':WD_E\rightarrow\SO_{2n-1}(\mathbb{C})$, which implies that $\theta_{\psi_0}^{n,n}(\tau)\neq0$ by \cite[Theorem 4.3]{AG}. It contradict the assumption $\theta_{\psi_0}^{n,n}(\tau)=0$.
	 Therefore
	\[\dim\Hom_{\Sp_{2n}(F)}(\tau,\mathbb{C})=\dim\Hom_{\Sp_{2n}(E)}(I_0(-1/2),\tau^\delta)\geq\dim\Hom_{\Sp_{2n}(E)}(I_0(-1/2),\tau^\delta) \]
	and so $\dim\Hom_{\Sp_{2n}(F)}(\tau,\mathbb{C})\geq \dim\Hom_{\Oo_{n+1,n+1}(F)}(\Sigma,\mathbb{C})$. Thus $$\dim\Hom_{\Sp_{2n}(F)}(\tau,\mathbb{C})=\dim\Hom_{\Oo_{n+1,n+1}(F)}(\Sigma,\mathbb{C}).$$
	
	\item If $\theta^{n,n}_{\psi_0}(\tau)$ is nonzero, then one can consider 
	\[\xymatrix{\Sp_{2n-2}(E)\ar@{-}[d]\ar@{-}[rd] &\Oo_{2n,2n}(F)\ar@{-}[d]\ar@{-}[rd]\ar@{-}[ld] &\Sp_{2n}(E)\ar@{-}[ld]\ar@{-}[d] \\ \Sp_{2n-2}(F)&\Oo_{n,n}(E)&\Sp_{2n}(F) } \]
	with $\tau=\Theta^{2n}_{\psi_0}(\Sigma)=\Theta_{\psi_E}^{2n}(\Sigma^\delta)$. There exists a short exact sequence of $\Oo_{2n,2n}(F)$-modules
	\[\xymatrix{0\ar[r]& R^{2n}(\mathbf{1})\ar[r]&\mathcal{I}(1/2)\ar[r]&R^{2n-2}(\mathbf{1})\otimes\det\ar[r]&0   } \]
where $R^{2r}(\mathbf{1})$ is the big theta lift to $\Oo_{2n,2n}(F)$ of the trivial representation from $\Sp_{2r}(F)$, $\det$ is the determinant map of $\Oo_{2n,2n}(F)$ and $\mathcal{I}(s)$ is the Siegel  degenerate principal series  of $\Oo_{2n,2n}(F)$. (See \cite[Proposition 7.2]{gan2014formal}.) Note that $$\Oo_{n,n}(E)\subset\SO_{2n,2n}(F).$$ Then there is an inequality
\[\dim\Hom_{\Oo_{n,n}(E)}(\mathcal{I}(1/2),\Sigma^\delta )\leq\dim\Hom_{\Sp_{2n}(F)}(\tau,\mathbb{C})+\dim\Hom_{\Sp_{2n-2}(F)}(\Theta^{2n-2}_{\psi_0}(\Sigma),\mathbb{C}),   \]
where $\Theta^{2n-2}_{\psi_0}(\Sigma)$ is the big theta lift to $\Sp_{2n-2}(E)$ of $\Sigma$ of $\Oo_{n,n}(E)$.

By the assumption $\Theta_{\psi_0}^{2n-2}(\Sigma)=0$, we can use the same idea to show that
\[\dim\Hom_{\Sp_{2n}(F)}(\tau,\mathbb{C})=\dim\Hom_{\Oo_{n,n}(E)}((R^{2n}(\mathbf{1})\otimes\det) \cap R^{2n}(\mathbf{1}),\Sigma^\delta )= \dim\Hom_{\Oo_{n,n}(E)}(\mathcal{I}(-1/2),\Sigma^\delta)  \]
which equals to the sum
\begin{equation}
\sum_V\dim\Hom_{\Oo(V)}(\Sigma,\mathbb{C}) 
\end{equation}
where $V$ runs over all $2n$-dimensional quadratic spaces over $F$ such that $\Oo(V\otimes_F E)=\Oo_{n,n}(E)$. Therefore, $\tau$ is $\Sp_{2n}(F)$-distinguished if and only if $\Sigma=\theta^{n,n}_{\psi_0}(\tau)$ is $\Oo(V)$-distinguished for some $V$.
\end{enumerate}
This finishes the proof.
\end{proof}
\begin{rem}
	If $\Theta^{2n-2}_{\psi_0}(\Sigma)$ is nonzero, then $\Sigma$ is $\Oo(V)$-distinguished implies that $\tau$ is $\Sp(W_n)$-distinguished. However, it is not obvious that whether  $\tau$ is $\Sp_{2n}(F)$-distinguished implies that $\Sigma$ is $\Oo(V)$-distinguished for some certain quadratic space $V$ over $F$.
\end{rem}
We give a conjecture to end this section.
	 \paragraph*{\textbf{Conjecture}} Assume that $E$ is a quadratic field extension over $F$.
	 Let $\tau$ be a supercuspidal representation of $\Sp_{2n}(E)$.
	 Assume that $\theta(\tau)\neq0$ for some $V_E$ with discriminant algebra a quadratic extension of $E$ which does not come from $F$, and that  $\theta(\tau)=0$ for all $V_E$ with discriminant algebra a quadratic extension of $E$ which comes from $F$ (in particular if the normalized discriminant is trivial).
	 Then $\tau$ is not $\Sp_{2n}(F)$-distinguished.

	 \begin{rem}
	 	Due to	the fact that the Prasad conjectures for the L-parameters for theta liftings in \cite{D} have been proved by Gan-Ichino in \cite[Appendix C]{gan2014formal}, we can reformulate the above conjecture in terms of the $L$-parameter of $\tau$. Suppose that the $L$-parameter $\phi_\tau=\sum_{i=1}^r \phi_i$  where $\phi_i$ are distinct irreducible orthogonal representations of $WD_E$ and that $\phi_\tau$ contains at least one quadratic character $\chi$ of $E^\times$. Assume that $\chi_i|_{F^\times}\neq\mathbf{1}$ for each $1$-dimensional summand $\phi_i=\chi_i$ (the quadratic characters of $E^\times$) of $\phi_\tau$. Then $\tau$ is not $\Sp_{2n}(F)$-distinguished.
  On the Galois side, there does not exist any extension of $\phi_\tau$ from $WD_E$ to $WD_F$ which can be seen as part of the conjecture of Dipendra Prasad in \cite[Conjecture 2]{prasad2015arelative} for $\Sp_{2n}$. 
	 \end{rem}
	 	 \begin{rem}
	 	In fact, we have studied the case when $\Oo(\Res_{E/F} V_E)\cong\Oo_{2n+2,2n-2}(F)$ in the proof of Theorem \ref{PrasadSpO} and have  verified the conjecture for $n\leq2$ in \cite{lu2018pacific,lu2019Sp(4)}. 
	 \end{rem}
\section{The $\Sp_4(F)$-distinguished representation of $\SU_{2,2}$}
In this section, we will show that there does not exist any tempered representation of $\SU_{2,2}(E/F)$ distinguished by $\Sp_4(F)$, where $E/F$ is a quadratic field extension. 

 Let $V$ be  an $n$-dimensional quadratic space over $F$ with $n\geq5$. Let 
\[G=\begin{cases}
\SL_2(F)&\mbox{if }n \mbox{ is even};\\
\Mp_2(F)&\mbox{if }n \mbox{ is odd}.
\end{cases} \]
We will consider the theta correspondence between the special orthogonal group $\SO(V)$ and $\SL_2(F)$ or $\Mp_2(F)$.
Fix a unitary nontrivial additive character $\psi$ of $F$. Let $\omega_\psi$ be the Weil representation of $\Mp_{2n}(F)$. Given an irreducible representation $\Sigma$ of $\SO(V)$, there exists a unique (genuine) representation $\tau$ of $G$ such that the maximal $\Sigma$-isotropic quotient of $\omega_\psi$ is of the form
\[\Sigma\boxtimes\tau, \]
where $\tau$ is called the big theta lift of $\Sigma$. 
Given a tempered representation $\Sigma$ of $\SO(V)$, assume that the small theta lift $\theta_\psi(\Sigma)$ is a nonzero representation of $G$. This means that the first occurence of $\Sigma$ is equal to $2$, i.e. $m^{\mbox{down}}(\Sigma)=2$ in the sense of \cite{AG}. We can use the theta correspondence for the pair $(\SO(V),G)$ instead of the dual pair $(\Oo(V),G)$ if $\dim V\geq3$. (See \cite[\S5]{D}.) Here $\dim V$ is very small, so the local theta correspondence for $(\SO(V),G)$ can be very explicit.
\begin{thm}
	Let $V'\subset V$ be  a subspace of $V$ of codimension one and $\dim_FV\geq5
	$. If there exists a tempered representation $\Sigma$ of $\SO(V)$ distinguished by $\SO(V')$, then the rank of the subgroup $\SO(V')$ is less or equal to $1$.
	Thus, the pair $(\SO(V),\SO(V'))$ sits inside the following table
\begin{table}[h]
		\renewcommand*{\arraystretch}{1.5}
	\begin{tabular}{|c|c|c|c|c|c|c|}
		\hline $\SO(V) $&$\SO_{3,2}(F)$&$\SO_{4,1}(F)$&$\SO_{4,1}(F)$&$\SO_{4,2}(F)$&$\SO_{5,1}(F)$ &$\SO_{5,2}(F)$\\
		\hline $\SO(V')$&$\SO_{3,1}(F)$&$\SO_{4,0}(F)$&$\SO_{3,1}(F)$&$\SO_{4,1}(F)$&$\SO_{4,1}(F)$&$\SO_{5,1}(F)$\\
		\hline
	\end{tabular}\end{table} 
\end{thm}
\begin{proof}
	Thanks to \cite[Proposition 9]{PD}, the assumption that $\Sigma$ is $\SO(V
	')$-distinguished implies that the small theta lift $\theta_\psi(\Sigma)$ from $\SO(V)$ to $G$ is nonzero and that the big theta lift $\tau$ of $\Sigma$ is a $\psi_a$-generic representation of $G$, where $\psi_a(x)=\psi(ax)$ is a nontrivial additive character of $F$ for $a\in F^\times/{F^\times}^2$. Since $\theta_\psi(\Sigma)\neq0$, we obtain $m^{\mbox{down}}(\Sigma)=2$. Due to the chain condition (defined in the Appendix) in \cite[Theorem 4.1]{AG}, there is an upper bound for the dimension of $V$, i.e. $\dim V=n\leq7$. If $n=\dim V>5$ or the Langlands parameter $\phi_\Sigma\neq 2S_2$(where $S_r$ is the irreducible algebraic representation $\mathrm{Sym}^{r-1}(\mathbb{C})$ of $\SL_2(\mathbb{C})$), then the Langlands parameter of $\Sigma$ contains $S_{n-3}$ with multiplicity one. \cite[Proposition 5.4]{AG} implies that the big theta lift $\tau=\Theta_\psi(\Sigma)$ is irreducible and tempered. So $\tau=\theta_\psi(\Sigma)$ and
	\begin{equation}\label{p-adicwhit}
	\dim\Hom_{N}(\tau,\psi_a)=\dim\Hom_{\SO(V^a)}(\Sigma,\mathbb{C}),
	\end{equation} 
	where $N\cong F$ is the unipotent subgroup in $G$ and $V^a$ is the subspace of $V$.
	If $n=5$ and $\phi_\Sigma=2S_2$, then $\theta_\psi(\Sigma)$ is a discrete series representation of $G$. Moreover, all irreducible subquotients of $\tau$ are tempered due to \cite[Proposition 5.5]{AG} and so $\tau=\theta_\psi(\Sigma)$. Furthermore, $m^{\mbox{up}}(\tau)=\dim V$. According to $n$, there are a few cases:
	\begin{enumerate}[(i).]
		\item The case $n=5$. There are two subcases.
		\begin{itemize}
			\item If $\SO(V)\cong\PGSp_4(F)$ is split, then the subgroup $\SO(V')$ may be $\SO_{2,2}(F)$ or $\SO_{3,1}(F)\cong \GL_2(E)^\natural/F^\times$, where
			\[\GL_2(E)^\natural=\{g\in\GL_2(E):\det(g)\in F^\times \}\cong\GSpin(V'). \]
			 If $\SO(V')\cong\SO_{2,2}(F)$ is split, then  $a=1$ and
			\[\dim\Hom_{\SO(V')}(\Sigma,\mathbb{C})=\dim\Hom_{N}(\tau,\psi). \]
			Since $\tau$ participates in the theta correspondence between $PD^\times$ and $\Mp_2(F)$, it is not $\psi$-generic. Otherwise, if $\tau$ is $\psi$-generic, then the theta lift to $\PGL_2(F)$ of $\tau$ will be nonzero and tempered, which implies that $\Sigma$ is not tempered. Hence the tempered representation $\Sigma$ can never be $\SO_{2,2}(F)$-distinguished. However, \eqref{p-adicwhit} implies that $\Sigma$ is $\SO_{3,1}(F)$-distinguished  if and only if $\tau$ is $\psi_a$-generic,
			 where $a=\epsilon\in F^\times\setminus NE^\times$ and $\SO(V^a)\cong\SO_{3,1}(F)$. Here the quadratic field extension $E/F$ depends on $\tau$.
			\item If $\SO(V)\cong\PGSp_{1,1}(F)$ (the pure inner form of $\PGSp_4(F)$) is not quasi-split, then $\tau$ could be a tempered principal series representation of $G$, which is both $\psi$-generic and $\psi_{\epsilon}$-generic. In this case, $\Sigma$ could be both $\SO_{4,0}(F)$-distinguished and $\SO_{3,1}(F)$-distinguished.
		\end{itemize}
		\item The case $n=6$. The theta lift from $\SL_2(F)$ to the split special orthogonal group $\SO_{3,3}(F)$ can never be a tempered representation. (See \cite[\S4]{AG}.) So we only study the cases $\SO(V)\cong\SO_{4,2}(F)$ or $\SO(V)\cong\SO_{5,1}(F)$, where $\SO_{4,2}(F)$ is the special orthogonal group (quasi-split) of a $6$-dimensional quadratic vector space with discriminant $E$ and Hasse invariant $+1$ and $\SO_{5,1}(F)$ is the special orthogonal group (non-split) of a $6$-dimensional quadratic vector space of trivial discriminant.
		\begin{itemize}
			\item If $\SO(V)\cong\SO_{5,1}(F)$, then $\tau$ could be an irreducible principal series $I(\chi)$ of $\SL_2(F)$. Then  $\Sigma$ is $\SO_{4,1}(F)$-distinguished with Langlands parameter
			\[\phi_\Sigma=\chi+\mathbf{1}+\chi^{-1}+S_3. \]
			\item If $\SO(V)\cong\SO_{4,2}(F)$, then the representation $\tau$ of $\SL_2(F)$ is not $\psi$-generic but $\psi_{\epsilon}$-generic. So $\Sigma$ can never be $\SO_{3,2}(F)$-distinguished. However, $\Sigma$ is $\SO_{4,1}(F)$-distinguished once $\theta_\psi(\Sigma)\neq0$.
		\end{itemize}
	\item  The case $n=7$. The theta lift from $\Mp_2(F)$ to the split group $\SO_{4,3}(F)=\SO(V)$ must be nontempered, where the discriminant of $V$ is trivial. So we only need to consider the case $\SO(V)\cong\SO_{5,2}(F)$. There is only one case that $\phi_\Sigma=S_2\oplus S_4$ and $\tau=\omega_\psi^-$ is the  $\psi$-generic odd Weil representation of $\Mp_2(F)$. Therefore, $\Sigma$ is $\SO_{5,1}(F)$-distinguished where $a=1$ and $\SO(V^a)\cong\SO_{5,1}(F)$.
	\end{enumerate}
In a short summary, the rank of $\SO(V')$ is less or equal to $1$. Therefore, we have finished the proof.
\end{proof}
\begin{rem}
	If $\dim_FV=4$ or $3$, then $\dim_FV'$ is $3$ or $2$ respectively and so the rank of $\SO(V')$ is less or equal to $1$ automatically. The case $\SO_3\hookrightarrow\SO_4$ has been investigated in \cite{hengfei2016new}. Thus we have proved \cite[Conjecture 2]{PD} in the case of non-archimedean fields.
\end{rem}
\begin{coro}Assume that $\dim V=6$ and $\SO(V)=\SO_{4,2}(F)$. There is an isogeny $\SU_{2,2}\rightarrow\SO_{4,2}$ with kernel $\{\pm1\}$. (See \cite[\S6]{PD}.) There are two isogeny diagrams $$\xymatrix{\SU_{2,2}\ar[r]&\SO_{4,2}&&\SU_{2,2}\ar[r]&\SO_{4,2}\\ \Sp_4\ar@{^{(}->}[u]\ar[r]&\SO_{3,2}\ar@{^{(}->}[u] &\mbox{and}&\Sp_{1,1}\ar@{^{(}->}[u]\ar[r]&\SO_{4,1}\ar@{^{(}->}[u]  } $$
	where
	$\Sp_{1,1}$ is the unique inner form of $\Sp_4$ defined over $F$. Then
	\begin{itemize}
		\item  There does not exist any tempered representation of $\SU_{2,2}(E/F)$ distinguished by $\Sp_4(F)$.
		\item There do exist  supercuspidal representations of $\SU_{2,2}(E/F)$ distinguished by $\Sp_{1,1}(F)$.
	\end{itemize}
\end{coro}
\begin{rem}
	It generalizes the result of Dijols and Prasad \cite{PD} that there does not exist any supercuspidal representation of $\SU_{2,2}(F)$ distinguished by $\Sp_4(F)$.
\end{rem}

\appendix
\section{The local theta correspondence of tempered representations and Langlands parameters }
The purpose of this appendix is to describe theta lifts of tempered representations in terms of the local Langlands correspondence. (See \cite{AG} for more details.)

Let $V$ be a $m$-dimensional quadratic space over $F$ with quadratic character $\chi_V$. Let $\pi$ be a tempered representation of $\Oo(V)$ with enhanced $L$-parameter $(\phi,\eta)$ and $\det\phi=\chi_V$. We can decompose 
\[\phi=m_1\phi_1\oplus\cdots m_r\phi_r\oplus\phi'\oplus{\phi'}^\vee \]
 where $\phi_1,\cdots,\phi_r$ are distinct irreducible orthogonal representations of $WD_F$, $m_i = m_\phi(\phi_i)$ is multiplicity of $\phi_i$ contained in $\phi$, and $\phi'$ is a sum of irreducible representations of $WD_F$ which are not orthogonal. If $\tau$ is a square-integrable representation, then $m_i=1$ for all $i$ and $\phi'=0$. Let $S_\phi$ be the component group of $\phi$.
Let $W$ be a $2n$-dimensional symplectic vector space over $F$. Define $\ell=m-2n-1$ and 
\[\kappa=\begin{cases}
1\iif \ell\mbox{ is odd};\\
2\other.
\end{cases} \]
Let $S_r$ be the irreducible algebraic representation $\mathrm{Sym}^{r-1}(\mathbb{C})$ of $\SL_2(\mathbb{C})$.
\begin{thm}
	[Atobe-Gan] Let $\pi$ be a tempered representation of $\Oo(V)$ with enhanced $L$-parameter $(\phi,\eta)$.
	\begin{enumerate}[(i)]
		\item 
	 Consider the set $\mathfrak{T}$ containing $\kappa-2$ and all integers $l>0$ with $\ell\equiv\kappa\pmod{2} $ satisfying the following conditions:
	 \begin{itemize}
	 	\item (chain condition) $\phi$ contains $S_r$ for $r=\kappa,\kappa+2,\cdots,\ell$;
	 	\item (odd-ness condition) the multiplicity $m_{\phi}(S_r)$ is odd for $r=\kappa,\kappa+2,\cdots,\ell-2$;
	 	\item (initial condition) if $\kappa=2$, then $\eta(e_2)=-1$ if $m$ is odd;
	 	\item (alternationg condition) $\eta(e_r)=\eta(e_{r+2})$ for $r=\kappa,\kappa+2,\cdots,\ell-2$.
	 \end{itemize}
	Here, $e_r$ is the element in $S_\phi$ corresponding to $S_r$. Let
$	\ell(\pi) = \max \mathfrak{T} .$ Then
	\[m^{\mbox{down}}(\pi)=m-1-\ell(\pi)\quad\mbox{  and  }\quad m^{\mbox{up}}(\pi)=m+1+\ell(\pi).  \]
	
	\item Assume that the theta lift $\theta(\pi)$ to $\Sp(W)$ or $\Mp(W)$ (depends on $n$) is nonzero and has an enhanced $L$-parameter $(\phi_{\theta(\pi)},\eta_{\theta(\pi)} )$. Put $m_1=m^{\mbox{down}}(\pi)+1-\kappa$. If $m^{\mbox{down}}(\pi)\leq \dim W_n< m_1$, then
	\[\phi_{\theta(\pi)}=(\phi\otimes\chi_V)-\chi_V S_\ell,  \]
	where $\ell=m-2n-1$. If $\dim W_n=m_1$, then
	\[\phi_{\theta(\pi)}=\begin{cases}
	\phi\otimes\chi_V\oplus\chi_V\iif \kappa=1;\\
	\phi\otimes\chi_V\iif \kappa=2.
	\end{cases} \]
	If $\dim W_n> m_1$, then $\phi_{\theta(\pi)}$ is not bounded.
	\item If $\dim W_n >m^{\mbox{up}}(\pi)$, then $\theta(\pi)$ is nonzero and $\phi_{\theta(\pi)}$ is not bounded. Therefore, $\theta(\pi)$ is not tempered.
\end{enumerate}
\end{thm}
The results of Atobe-Gan in \cite{AG} are very general. There are relations between the characters $\eta$ and $\eta_{\theta(\pi)}$ which are omitted here due to the purpose that this paper focuses on the relations between Langlands parameters $\phi$ and $\phi_{\theta(\pi)}$. There are analogous theorems if we switch $\Sp(W)$ and $\Oo(V)$.


\bibliographystyle{amsalpha}             
\bibliographystyle{plain}
\bibliography{Sp(4R)}

\providecommand{\bysame}{\leavevmode\hbox to3em{\hrulefill}\thinspace}
\providecommand{\MR}{\relax\ifhmode\unskip\space\fi MR }
\providecommand{\MRhref}[2]{%
  \href{http://www.ams.org/mathscinet-getitem?mr=#1}{#2}
}
\providecommand{\href}[2]{#2}
\begin{thebibliography}{\'{O}87}

\bibitem[AC17]{choiy2017}
Mahdi Asgari and Kwangho Choiy, \emph{The local {L}anglands conjecture for
  {$p$}-adic {$\rm GSpin_4$}, {$\rm GSpin_6$}, and their inner forms}, Forum
  Math. \textbf{29} (2017), no.~6, 1261--1290. \MR{3719299}

\bibitem[AG17]{AG}
H.~Atobe and W.~T. Gan, \emph{Local theta correspondence of tempered
  representations and {L}anglands parameters}, Invent. Math. \textbf{210}
  (2017), no.~2, 341--415. \MR{3714507}

\bibitem[AP19]{adler2018multiplicity}
Jeffrey~D Adler and Dipendra Prasad, \emph{Multiplicity upon restriction to the
  derived subgroup}, to appear in Pac. J. Math. (2019).

\bibitem[DP18]{PD}
Sarah Dijols and Dipendra Prasad, \emph{Symplectic models for unitary groups},
  to appear in Trans. Amer. Math. Soc. (2018).

\bibitem[Gan18]{gan2018period}
W.~T. Gan, \emph{Periods and theta correspondence}, in preprint (2018).

\bibitem[GGP12]{GGP1}
W.~T. Gan, B.~H. Gross, and D.~Prasad, \emph{Symplectic local root numbers,
  central critical {$L$} values, and restriction problems in the representation
  theory of classical groups}, Ast\'erisque (2012), no.~346, 1--109, Sur les
  conjectures de Gross et Prasad. I. \MR{3202556}

\bibitem[GI14]{gan2014formal}
W.~T. Gan and A.~Ichino, \emph{Formal degrees and local theta correspondence},
  Invert. Math. \textbf{195} (2014), no.~3, 509--672.

\bibitem[GS15]{GS}
Nadya Gurevich and Dani Szpruch, \emph{The non-tempered {$\theta_{10}$}
  {A}rthur parameter and {G}ross-{P}rasad conjectures}, J. Number Theory
  \textbf{153} (2015), 372--426. \MR{3327582}

\bibitem[GSS18]{dima2018analytic}
D.~Gourevitch, S.~Sahi, and E.~Sayag, \emph{Analytic continuation of
  equivariant distributions}, International Mathematics Research Notices
  (2018), rnx326.

\bibitem[GT11]{gan2011locallanglands}
W.~T. Gan and S.~Takeda, \emph{The local {L}andlands conjecture for {$\rm
  GSp(4)$}}, Ann. Math. \textbf{173} (2011), no.~3, 1841--1882.

\bibitem[GT16a]{gan2014howe}
\bysame, \emph{On the {H}owe duality conjecture in classical theta
  correspondence}, Advances in the theory of automorphic forms and their
  {$L$}-functions, Contemp. Math., vol. 664, Amer. Math. Soc., Providence, RI,
  2016, pp.~105--117.

\bibitem[GT16b]{gan2014proof}
\bysame, \emph{A proof of the {H}owe duality conjecture}, J. Amer. Math. Soc.
  \textbf{29} (2016), no.~2, 473--493.

\bibitem[KR05]{kudla2005first}
S.~Kudla and S.~Rallis, \emph{On first occurrence in the local theta
  correspondence}, Automorphic representations, {$L$}-functions and
  applications: progress and prospects, Ohio State Univ. Math. Res. Inst.
  Publ., vol.~11, de Gruyter, Berlin, 2005, pp.~273--308.

\bibitem[Kud96]{kudla1996notes}
S.~Kudla, \emph{Notes on the local theta correspondence}, unpublished notes,
  available online (1996).

\bibitem[Lu17a]{hengfei2017}
H.~Lu, \emph{$\mathrm{GSp(4)}$-period problems over a quadratic field
  extension}, Ph.D. thesis, National University of Singapore, 2017.

\bibitem[Lu17b]{hengfei2016new}
\bysame, \emph{A new proof to the period problems of {$\rm GL(2)$}}, J. Number
  Theory \textbf{180} (2017), 1--25.

\bibitem[Lu18]{lu2018pacific}
\bysame, \emph{Theta correspondence and the {P}rasad conjecture for $\rm
  {SL(2)}$}, Pac. J. Math. \textbf{295} (2018), no.~2, 477--498.

\bibitem[Lu19]{lu2019Sp(4)}
\bysame, \emph{The {P}rasad conjectures for $\mathrm{U}_2$,$\mathrm{SO}_4$ and
  $\mathrm{Sp}_4$}, arXiv preprint arXiv:1903.11793 (2019).

\bibitem[\'{O}87]{olafsson1987fourier}
G.~\'{O}lafsson, \emph{Fourier and {P}oisson transformation associated to a
  semisimple symmetric space}, Invent. Math. \textbf{90} (1987), no.~3,
  605--629. \MR{914851}

\bibitem[Pra93]{D}
Dipendra Prasad, \emph{On the local {H}owe duality correspondence}, Internat.
  Math. Res. Notices (1993), no.~11, 279--287. \MR{1248702}

\bibitem[Pra96]{prasad1996branching}
\bysame, \emph{Some applications of seesaw duality to branching laws}, Math.
  Ann. \textbf{304} (1996), no.~1, 1--20. \MR{1367880}

\bibitem[Pra15]{prasad2015arelative}
\bysame, \emph{A 'relative' local {L}anglands correspondence}, arXiv preprint
  arXiv:1512.04347 (2015).

\bibitem[Rob96]{roberts1996similitudes}
Brooks Roberts, \emph{The theta correspondence for similitudes}, Israel J.
  Math. \textbf{94} (1996), 285--317. \MR{1394579}

\bibitem[SV17]{SV2017}
Yiannis Sakellaridis and Akshay Venkatesh, \emph{Periods and harmonic analysis
  on spherical varieties}, Ast\'erisque (2017), no.~396, viii+360. \MR{3764130}

\bibitem[SZ15]{sun2012conservation}
B.~Sun and C.-B. Zhu, \emph{Conservation relations for local theta
  correspondence}, J. Amer. Math. Soc. \textbf{28} (2015), no.~4, 939--983.

\bibitem[Wal90]{waldspurger1990demonstration}
J.-L. Waldspurger, \emph{D\'emonstration d'une conjecture de dualit\'e de
  {H}owe dans le cas {$p$}-adique, {$p\neq 2$}}, Festschrift in honor of {I}.
  {I}. {P}iatetski-{S}hapiro on the occasion of his sixtieth birthday, {P}art
  {I} ({R}amat {A}viv, 1989), Israel Math. Conf. Proc., vol.~2, Weizmann,
  Jerusalem, 1990, pp.~267--324. \MR{1159105}

\bibitem[Wal12]{wd2012so}
\bysame, \emph{La conjecture locale de {G}ross-{P}rasad pour les
  repr\'esentations temp\'er\'ees des groupes sp\'eciaux orthogonaux},
  Ast\'erisque (2012), no.~347, 103--165, Sur les conjectures de Gross et
  Prasad. II. \MR{3155345}

\bibitem[Zha18]{zhang2018theta}
Chong Zhang, \emph{Theta lifts and distinction for regular supercuspidal
  representations}, arXiv preprint arXiv:1804.09878 (2018).

\end{thebibliography}
\end{document}